\documentclass[11pt,leqno]{amsart}
\usepackage{amsmath,amssymb,amsthm}
\newtheorem{theorem}{Theorem}[section]
\newtheorem{lemma}[theorem]{Lemma}

\newtheorem{proposition}[theorem]{Proposition}
\newtheorem{corollary}[theorem]{Corollary}
\theoremstyle{definition}

\theoremstyle{remark}

\numberwithin{equation}{section}
\def\fnote#1{\footnote}

\def\natu{{\mathbb N}}

\def\real{{\mathbb R}}

\def\ignora#1{}
\def\n3#1{\left\vert  \! \left\vert \! \left\vert \, #1 \, \right\vert \!
  \right\vert \! \right\vert }

\begin{document}

\title{ Big slices versus big relatively weakly open subsets in Banach spaces }
%\author{}
%\address{Universidad de Granada, Facultad de Ciencias.
%Departamento de Matem\'{a}tica Aplicada, 18071-Granada (Spain)}

\author{Julio Becerra Guerrero, Gin{\'e}s L{\'o}pez-P{\'e}rez and Abraham Rueda Zoca}
\address{Universidad de Granada, Facultad de Ciencias.
Departamento de An\'{a}lisis Matem\'{a}tico, 18071-Granada
(Spain)} \email{glopezp@ugr.es, juliobg@ugr.es}

\thanks{The first author was partially supported by MEC (Spain) Grant MTM2011-23843 and Junta de Andaluc\'{\i}a grants
FQM-0199, FQM-1215. The second author was partially supported by
MEC (Spain) Grant MTM2012-31755 and Junta de Andaluc\'{\i}a Grant
FQM-185.} 
\maketitle \markboth{J. Becerra, G. L\'{o}pez and A. Rueda   }{
  Slices versus weakly open sets }

\begin{abstract}
We study the unknown differences between the size of slices and
relatively weakly open subsets of the unit ball in Banach spaces.
We show that every Banach space containing $c_0$ isomorphically
satisfies that every slice of its unit ball has diameter 2  so
that its unit ball contains nonempty relatively weakly open
subsets with diameter arbitrarily small, which answer an open
question and stresses the differences between the size of slices
and relatively weakly open subsets of the unit ball of Banach
spaces.

\end{abstract}

\section{Introduction}
\par
\bigskip

The well known Radon-Nikodym (RNP) property in Banach spaces is
characterized by the existence of slices with diameter arbitrarily
small in every closed and bounded subset of the space. Similarly,
a Banach space $X$ has the point of continuity property (PCP) if
every nonempty closed and bounded subset of $X$ has relatively
weakly open subsets with diameter arbitrarily small. We refer to
\cite{BB}, \cite{GGMS} and \cite{GMS} for background about RNP and
PCP. It is clear then that RNP implies PCP, however there are
Banach spaces satisfying PCP and failing RNP \cite{BR}. In the
last years, one can find what we can call the big slice phenomena,
that is, examples of Banach spaces where every slice or every
nonempty relatively weakly open subset of its unit ball has
diameter 2, a property extremely opposite to RNP or PCP. These
examples include infinite-dimensional uniform algebras
\cite{NyWe}, infinite-dimensional $C^*$-algebras \cite{BLR},
infinite-dimensional M-embedded spaces \cite{LoPe},   Banach
spaces with the Daugavet property \cite{Sh}, etc. Also, it is
known \cite[Lemma I.1.3]{DGZ} that all these spaces have extremely
rough dual norm, which is a property extremely opposite to the
Fr{\'e}chet differentiability. The big slice phenomena probably
started in the paper of O. Nygaard and D. Werner \cite{NyWe}, but
after discover many examples with this phenomena and the
connections with other well known geometrical properties, like
Daugavet property or extreme roughness, the phenomena has now its
own life and new geometrical properties in Banach spaces have
appeared, the slice diameter 2 property and the diameter 2
property. Recall now the precise definitions.

Given a Banach space $X$, we say that $X$ has the slice diameter 2
property (SD2P) if every slice of the unit ball of $X$ has
diameter 2. Similarly, we say that $X$ has the diameter 2 property
(D2P) if every nonempty relatively weakly open subset of the unit
ball of $X$ has diameter 2.

With the above definitions it is clear that SD2P implies non-RNP
and D2P implies non-PCP. As RNP and PCP are isomorphic properties,
that is, they are independent of the equivalent norm considered in
the space, one can see SD2P and D2P like the extremely opposite
geometrical properties to RNP or PCP, since SD2P and D2P are not
independent of the norm considered in the space.

From the definitions, one deduce that D2P implies SD2P. In fact
all known Banach spaces  with the SD2P up to now also satisfying
D2P. It is then an open problem wether these two properties are in
fact different. As RNP and PCP are different properties, it is
natural thinking that SD2P and D2P are also different properties.
However, the well known example of Banach space with PCP and
failing RNP is $B$, the natural predual of James tree space $JT$,
constructed in \cite{J}, and it is proved in the paper by W.
Schachermayer, A. Sersouri and E. Werner \cite{ScSeWe} that $B$
fails the SD2P. So the natural candidate to example of Banach
space with SD2P and failing D2P doesn't work.

The aim of this note is to prove the existence of a Banach space
satisfying SD2P and failing D2P, which answers by the negative an
open problem stated firstly in \cite{ALN}. In fact, much more can
be shown. We prove in Theorem \ref{c0} that every Banach space
containing isomorphically $c_0$, the classical Banach space of
null sequences with the sup norm,   can be equivalently renormed
satisfying SD2P and so that its unit ball contains nonempty
relatively weakly open subsets with diameter arbitrarily small. As
a consequence, every Banach space containing isomorphic copies of
$c_0$ can be equivalently renormed satisfying SD2P and failing
D2P. For this, we first construct in Proposition \ref{K} and
Theorem \ref{Kw} a closed, bounded, convex and symmetric subset of
$c$, the Banach space of convergent sequences with the sup norm,
so that every slice of $K$ has diameter 2 and its unit ball
contains nonempty relatively weakly open subsets with diameter
arbitrarily small. Finally, we get in Corollary \ref{corfin} that
the $\ell_p$-sum of Banach spaces satisfying SD2P and failing D2P
also satisfies SD2P and fails D2P.

We pass now to introduce some notation. For a Banach space $X$,
$X^*$ denotes the topological dual of $X$,  $B_X$ and $S_X$ stand
for the closed unit ball and unit sphere of $X$, respectively, and
$w$ denotes the weak topology in $X$. We consider only real Banach
spaces. A slice of a set $C$ in $X$ is a set of $X$ given by
$$S=\{x\in C :x^*(x)>\sup x^*(C) -\alpha \}$$
where $x^*\in X^*$ and $\alpha  <\sup x^*(C)$.

Recall that a slice of $B_X$ is a nonempty relatively weakly open
subset of $B_X$ and the family
$$\{\{x\in B_X: \vert x_i^*(x-x_0)\vert<\varepsilon,\ 1\leq i\leq
n\}:n\in\natu ,\ x_1^*,\cdots ,x_n^*\in X^*\}$$ is a basis of
relatively weakly open neighborhoods of $x_0\in B_X$. So every
relatively weakly open subset of $B_X$ has nonempty intersection
with $S_X$, whenever $X$ has infinite dimension.

$\natu^{<\omega}$ stands for the set of all ordered finite
sequences of positive integers and denotes by $0$ the empty
sequence. If $\alpha=(\alpha_1,\cdots ,\alpha_n)\in
\natu^{<\omega}$, we define the length of $\alpha$ by $\vert
\alpha\vert =n$ and $\vert 0\vert =0$. Also we use the natural
order in $\natu^{<\omega}$ given by:
$$\alpha\leq \beta \ {\rm if}\ \vert \alpha\vert \leq \vert
\beta\vert\ {\rm and}\  \alpha_i=\beta_i\ \forall i\in \{1,\cdots
,\vert \alpha\vert\}.$$ Also we do $0\leq \alpha\ \forall
\alpha\in \natu^{<\omega}$.

As $\natu^{<\omega}$ is a countable set we can construct a
bijective map $\phi :\natu^{<\omega}\rightarrow \natu$ so that
$\phi (0)=1$ and $\phi(\alpha)\leq\phi(\beta)$ whenever
$\alpha\leq \beta\in \natu^{<\omega}$ and $\phi(\alpha, j)\leq
\phi(\alpha ,k)$ for every $\alpha\in \natu^{<\omega}$ and $j\leq
k\in\natu$. Indeed, consider $\{p_n\}$ an enumeration of prime
positive integers numbers and define the bijective map
$\phi_0:\natu^{<\omega}\rightarrow \natu$ given by
$\phi_0(\alpha_1,\ldots ,\alpha_k)=p_1^{\alpha_1}\cdots
p_k^{\alpha_k}.$ Now take an strictly increasing map
$\phi_1:\phi_0(\natu^{<\omega})\rightarrow \natu$ and put
$\phi=\phi_0\circ\phi_1$. Then $\phi$ satisfies the desired
properties. Observe that, from the above construction,
$\{\phi(\alpha ,j)\}_j$ is a strictly increasing sequence   for
every $\alpha\in\natu^{<\omega}$.

\section{Main result}
\par
\bigskip

We begin constructing a subset $A$ of $c$, the space of convergent
scalar sequences with the sup norm. For this, $\{e_n\}$ and
$\{e_n^*\}$ stand for the usual basis and the sequence of
biorthogonal functionals of $c_0$, the space of null scalars
sequences with the sup norm. Define for every $\alpha\in
\natu^{<\omega}$, $e_{\alpha}=:e_{\phi(\alpha)}\in c$,
$e_{\alpha}^*=:e_{\phi(\alpha)}^*\in c^*$ and $x_{\alpha}\in c$ by
$x_{\alpha}(i)=1$ if $\phi^{-1}(i)\leq \alpha$ and
$x_{\alpha}(i)=-1$ in otherwise. It is clear that $x_{\alpha}\in
S_c$ for every $\alpha\in \natu^{<\omega}$. Note that if $\alpha,\
\beta\in\natu^{<\omega}$ are incomparable then $\Vert
x_{\alpha}-x_{\beta}\Vert_c=2$.

Define $A=\{x_{\alpha}:\alpha\in\natu^{<\omega}\}$, which is a
subset of the unit sphere of $c$ and $K=\overline{\rm co}(A\cup
-A)$ which is a closed, convex and symmetric subset of $B_c$ with
diameter 2. Throughout this note, the aforementioned elements
$e_{\alpha}$, $e_{\alpha}^*$, $x_{\alpha}$ and the sets $A$ and
$K$ will be used without previous notice.

The first step is to prove that every slice of $K$ has diameter 2.

\begin{proposition}\label{K} Every slice of $K$ has diameter 2, as a subset
of $c$.\end{proposition}
\begin{proof} Pick $x^*\in S_{c^*}$, $\lambda <\sup x^*(K) $ and put
$S=\{x\in K:x^*(x)>\sup x^*(K) -\lambda\}$. As $S$ is a slice of
$K$ and $K=\overline{\rm co}(A\cup -A)$, we deduce that $S$ meets
$A$ or $S$ meets $-A$. From the symmetry of $K$we can assume that
$S\cap A\neq\emptyset$. Then there is $\alpha\in\natu^{<\omega}$
such that $x_{\alpha}\in S$. Pick $j\in\natu$. Then $x_{(\alpha,
j)}$ is an element in $A$ given by $x_{(\alpha, j)}(i)=1$ if
$\phi^{-1}(i)\leq (\alpha, j)$ and $x_{(\alpha, j)}(i)=-1$ in
otherwise. Hence $\{x_{(\alpha, j)}\}_j$ is a sequence in
$A\subset K$ weakly convergent to $x_{\alpha}$. So there is
$j\in\natu$ such that $x_{(\alpha ,j)}\in S$ and $$diam(S)\geq
\Vert x_{(\alpha, j)}-x_{\alpha}\Vert\geq\vert x_{(\alpha,
j)}(\phi((\alpha, j)))-x_{\alpha}(\phi((\alpha, j)))\vert=\vert
1-(-1)\vert=2.$$ Recalling that $K$ has diameter 2, we deduce that
$S$ has diameter 2, being $S$ any slice of $K$. \end{proof}

Now, we prove that $K$, as a subset of $c$ has relatively weakly
open subsets with diameter arbitrarily small.

\begin{proposition}\label{Kw} Given $n\in\natu$ and $\rho>0$ with
$\rho <\frac{1}{n(32n-21)}$, one has that $diam(W_n)<\frac{9}{n}$,
where $W_n$ is the relative weak open subset of $K\subset c$ given
by
$$W_n=\{x\in K:e_{(0,i)}^*(x)>\frac{1}{n}-1-\rho,\ 1\leq i\leq n,\
\lim_n x(n)<-1+\rho\}.$$\end{proposition}

\begin{proof} First of all, check that $x_0=\sum_{i=1}^n
\frac{x_{(0,i)}}{n}\in W_n$. For this note that $\lim_n
x_{(0,i)}(n)=-1$ and then $\lim_n x_0(n)=-1<-1+\rho$. Furthermore
$x_0$ is a convex combination of elements of $A\subset K$ and so
$x_0\in K$. Finally, for $ 1\leq j\leq n$, one has
$$e_{(0,j)}^*(x_0)=\sum_{i=1\ i\neq j}^n
\frac{1}{n}e_{(0,j)}^*(x_{(0,i)})+\frac{1}{n}
e_{(0,j)}^*(x_{(0,j)})=-\frac{n-1}{n}+\frac{1}{n}>\frac{1}{n}-1-\rho.$$

Then $W_n$ is nonempty. In order to prove that $diam
(W_n)<\frac{9}{n}$, it is enough to see that $diam(W_n\cap
co(A\cup -A))<\frac{9}{n}$. For this, pick  arbitraries $x, x'\in
co(A\cup -A)$, hence there are $\lambda, \lambda'\in (0,1]$, $a,\
a',\ -b,\ -b'\in co(A)$ such that $x=\lambda a+(1-\lambda)b$ and
$x'=\lambda' a'+(1-\lambda')b'$. Now $\lim_n a(n)=\lim_n a'(n)=-1$
and $\lim_n b(n)=\lim_n b'(n)=1$. As $x\in W_n$, we have that
$\lim_n \lambda a(n)+(1-\lambda)b(n)<-1+\rho$ and then we get that
\begin{equation}\label{landa}
2(1-\lambda)<\rho. \end{equation}
 Similarly we get that
\begin{equation}\label{landaprima}
2(1-\lambda')<\rho,
\end{equation}
and so
\begin{equation}\label{landa-landaprima}
\vert\lambda -\lambda'\vert  <\rho/2. \end{equation}
 For
$i\in\{1,\cdots, n\}$ one has, taking into account \ref{landa},
\ref{landaprima} and the fact $x=\lambda a+(1-\lambda)b\in W_n$
that
$$e_{(0,i)}^*(a)>\frac{\frac{1}{n}
-1-\rho-(1-\lambda)e_{(0,i)}^*(b)}{\lambda}>
\frac{\frac{1}{n}-1-\rho-\rho/2}{\lambda}.$$ It follows that

\begin{equation}\label{a} e_{(0,i)}^*(a)>(\frac{1}{n}-1-\frac{3\rho}{2})(1-\frac{\rho
}{2})^{-1}.
\end{equation}
Similarly, one gets that
\begin{equation}\label{aprima}
e_{(0,i)}^*(a')>(\frac{1}{n}-1-\frac{3\rho}{2})(1-\frac{\rho
}{2})^{-1}.
\end{equation}

Now, applying \ref{landa}, \ref{landaprima} and
\ref{landa-landaprima}, we have that
\begin{equation}\label{primera}
\Vert x-x'\Vert\leq \Vert\lambda a-\lambda'
a'\Vert+\Vert(1-\lambda)b-(1-\lambda')b'\Vert\leq \Vert\lambda
a-\lambda' a'\Vert+\rho\leq
\end{equation}
$$
\lambda\Vert a-a'\Vert+\vert\lambda-\lambda'\vert+\rho\leq\ \Vert
a-a'\Vert+\frac{3\rho}{2}.
$$

Now, our goal is estimate $\Vert a-a'\Vert$. For this put
$a=\sum_{j=1}^{p}\lambda_j x_{\alpha_j}$ and
$a'=\sum_{j=1}^q\beta_j x_{\alpha_{j}^{'}}$, where $p,q\in \natu$,
$\lambda_j,\beta_j>0$,
$\sum_{j=1}^p\lambda_j=\sum_{j=1}^q\beta_j=1$ and
$x_{\alpha_j},x_{\alpha_{j}^{'}}\in  A$.

Denotes by ${\bf 1}$ the sequence in $c$ with all its coordinates
equal 1, then $\Vert a-a'\Vert =\Vert a+{\bf 1}-(a'+{\bf
1})\Vert.$ Now $$ a+{\bf 1} =\sum_{j=1}^{p}\lambda_j
x_{\alpha_j}+{\bf 1}=\sum_{j=1}^p\lambda_j (x_{\alpha_j}+{\bf
1})=\sum_{j=1}^p\lambda_j \hat{x}_{\alpha_j},$$ where
$\hat{x}_{\alpha_j}$ is the element in $c$ given by
$\hat{x}_{\alpha_j}(i)=2$ if $\phi^{-1}(i)\leq \alpha_j$ and
$\hat{x}_{\alpha_j}(i)=0$ in otherwise. Similarly $a'+{\bf
1}=\sum_{j=1}^q\beta_j\hat{x}_{\alpha_j^{'}}$, where
$\hat{x}_{\alpha_j^{'}}$ is the element in $c$ given by
$\hat{x}_{\alpha_j^{'}}(i)=2$ if $\phi^{-1}(i)\leq\alpha_j^{'}$
and $\hat{x}_{\alpha_j^{'}}(i)=0$ in otherwise.

Now we have from \ref{a} and \ref{aprima} that
\begin{equation}\label{agorros}
e_{(0,j)}^*(\hat{a}),\
e_{(0,j)}^*(\hat{a}')>(\frac{1}{n}-2\rho)(1-\frac{\rho }{2})^{-1}\
\forall \ 1\leq i\leq n , \end{equation}
 where $\hat{a}=a+{\bf 1}$ and
$\hat{a}'=a'+{\bf 1}$.

For every $i\in\{1,\cdots ,n\}$ we define now
$$A_i=\{j\in\{1,\cdots ,p\}:\alpha_j\geq (0,i)\},\
A_i^{'}=\{j\in\{1,\cdots ,q\}:\alpha_j^{'}\geq (0,i)\}.$$ If
$i\neq k$ then $\alpha_i$ and $\alpha_k$ are incomparable and also
$\alpha_i^{'}$ and $\alpha_k^{'}$ are incomparable, hence $A_i\cap
A_k=\emptyset$ and $A_i^{'}\cap A_k^{'}=\emptyset$.

Now we have that from \ref{agorros} that
\begin{equation}\label{sumAi}
\sum_{j\in A_i}\lambda_j>(\frac{1}{n}-2\rho)(1-\frac{\rho
}{2})^{-1} ,\ \sum_{j\in
A_i^{'}}\beta_j>(\frac{1}{n}-2\rho)(1-\frac{\rho }{2})^{-1}.
\end{equation}

Then $$1=\sum_{j=1}^p\lambda_j=\sum_{j\in\cup_{i=1}^n
A_i}\lambda_j+\sum_{j\in(\cup_{i=1}^n
A_i)^c}\lambda_j=\sum_{i=1}^{n}\sum_{j\in
A_i}\lambda_j-\sum_{j\in(\cup_{i=1}^n A_i)^c}\lambda_j ,$$ and we
deduce from \ref{sumAi} that , for every $k\in\{1,\cdots ,n\}$,
\begin{equation}\label{sumAk}
\sum_{j\in A_k}\lambda_j=1-\sum_{i=1\ i\neq k}\sum_{j\in
A_i}\lambda_j-\sum_{j\in(\cup_{i=1}^n A_i)^c}\lambda_j<\break
1-\sum_{i=1\ i\neq k}^n
\frac{1}{n}-\frac{3\rho}{2}=
\end{equation}
$$
1-(n-1)(\frac{1}{n}-2\rho)(1-\frac{\rho }{2})^{-1}.
$$

Similarly
\begin{equation}\label{sumAkprima}
 \sum_{j\in
A_k^{'}}\beta_j<1-(n-1)(\frac{1}{n}-2\rho)(1-\frac{\rho
}{2})^{-1}.
\end{equation}

Also, from \ref{sumAi}
\begin{equation}\label{complemento}
\sum_{j\in(\cup_{i=1}^n A_i)^c}\lambda_j,\ \sum_{j\in(\cup_{i=1}^n
A_i^{'})^c}\beta_j<(2n-\frac{1}{2})\rho (1-\frac{\rho}{2} )^{-1}.
\end{equation}

Observe that  the vectors $\sum_{j\in
A_i}\lambda_j\hat{x}_{\alpha_j}-\sum_{j\in
A_i^{'}}\beta_j\hat{x}_{\alpha_{j}^{'}}$ have disjoint supports
for coordinates $k>1$ and $1\leq i\leq n$, and so
\begin{equation}\label{soportes}
max_{k>1}\vert(\sum_{i=1}^n(\sum_{j\in A_i}\lambda_j
\hat{x}_{\alpha_j}-\sum_{j\in
A_i^{'}}\beta_j\hat{x}_{\alpha_j^{'}}))(k)\vert\leq 2 max_{1\leq
i\leq n}\{\sum_{j\in A_i}\lambda_j +\sum_{j\in A_i^{'}}\beta_j\} ,
\end{equation}
since $\Vert \hat{x}_{\alpha_j}\Vert=\Vert
\hat{x}_{\alpha_j^{'}}\Vert=2.$

Now, applying \ref{complemento} and \ref{soportes}, we get
$$\Vert
\hat{a}-\hat{a}^{'}\Vert=\Vert \sum_{j=1}^p\lambda_j
\hat{x}_{\alpha_j}-\sum_{j=1}^q\beta_j\hat{x}_{\alpha_j^{'}}\Vert=$$
$$max\{\vert(\sum_{j=1}^p\lambda_j
\hat{x}_{\alpha_j}-\sum_{j=1}^q\beta_j\hat{x}_{\alpha_j^{'}})(k)\vert:k\in\natu
\}=$$
$$max\{\vert(\sum_{j=1}^p\lambda_j
\hat{x}_{\alpha_j}-\sum_{j=1}^q\beta_j\hat{x}_{\alpha_j^{'}})(\phi(0))\vert,\
\vert (\sum_{j=1}^p\lambda_j
\hat{x}_{\alpha_j}-\sum_{j=1}^q\beta_j\hat{x}_{\alpha_j^{'}})(k)\vert:k>1\}=$$
$$max\{0,\vert(\sum_{j=1}^p\lambda_j
\hat{x}_{\alpha_j}-\sum_{j=1}^q\beta_j\hat{x}_{\alpha_j^{'}})(k)\vert:k>1\}\leq$$
$$max\{\vert(\sum_{i=1}^n(\sum_{j\in A_i}\lambda_j
\hat{x}_{\alpha_j}-\sum_{j\in
A_i^{'}}\beta_j\hat{x}_{\alpha_j^{'}}))(k)+(\sum_{j\in(\cup_{i=1}^n
A_i)^c}\lambda_j\hat{x}_{\alpha_j})(k)-$$
$$(\sum_{j\in(\cup_{i=1}^n
A_i^{'})^c}\beta_j\hat{x}_{\alpha_j^{'}})(k)\vert:k>1\}\leq$$
$$max\{\vert(\sum_{i=1}^n(\sum_{j\in A_i}\lambda_j
\hat{x}_{\alpha_j}-\sum_{j\in
A_i^{'}}\beta_j\hat{x}_{\alpha_j^{'}}))(k)\vert:k>1\}+$$
$$+2(\sum_{j\in(\cup_{i=1}^n A_i)^c}\lambda_j
+\sum_{j\in(\cup_{i=1}^n A_i^{'})^c}\beta_j)\leq$$
$$2max_{1\leq i\leq n}\{\sum_{j\in
A_i}\lambda_j+\sum_{j\in A_i^{'}}\beta_j\}+4(2n-\frac{1}{2})\rho
(1-\frac{\rho}{2} )^{-1}\leq$$
$$4(1-(n-1)(\frac{1}{n}-2\rho)(1-\frac{\rho }{2})^{-1})
+4(2n-\frac{1}{2})\rho (1-\frac{\rho}{2} )^{-1}=$$
$$4(\frac{1}{n}+(4n-3)\rho )(1-\frac{\rho }{2})^{-1}.$$

Finally, we conclude  from \ref{primera} and the above estimation
that
$$\Vert x-x'\Vert\leq\Vert a-a'\Vert+3\rho/2\leq$$
$$4(\frac{1}{n}+(4n-3)\rho )(1-\frac{\rho }{2})^{-1}+\frac{3\rho}{2}= (\frac{8}{n}+(32n-21)\rho -\frac{3\rho^2}{2})(2-\rho)^{-1}<$$
$$ \frac{8}{n}+(32n-21)\rho < \frac{9}{n}$$
since $\rho <\frac{1}{n(32n-21)}$. Hence we have proved that $diam
(W_n)< \frac{9}{n}.$
\end{proof}

The above results find a closed, bounded, convex and symmetric
subset $K$ of $B_c$ satisfying that every slice of $K$ has
diameter 2 and $K$ contains nonempty relatively weakly open sets
with diameter arbitrarily small. Our next goal is showing how one
can get a Banach space whose unit ball behaves like  $K$ with
respect the size of slices and relatively weakly open subsets. For
this we need the following

\begin{lemma}\label{lemac0} Let $X$ be a Banach space containing an isomorphic
copy of $c_0$. Then there is an equivalent norm $| \|
 \cdot\| | $ in $X$ satisfying that $(X
,|\|\cdot\| |)$ contains an isometric copy of $c$ and for every
$x\in B_{(X,|\|\cdot\| |)}$ there are sequences $\{x_n\}$,
$\{y_n\}\in B_{(X,|\|\cdot\| |)}$ weakly convergent to $x$ such
that $ | \| x_n -y_n\| | =2$ for every $n\in \natu$. In fact,
$x_n=x+(1-\alpha_n)e_n$ and $y_n=x-(1+\alpha_n)e_n$ for some
scalars sequence $\{\alpha_n\}$ with $\vert \alpha_n\vert\leq 1$
for every $n$.  \end{lemma}

\begin{proof}As $X$ contains isomorphic copies of $c$, we can assume that $c$ is, in fact,
 an isometric subspace of $X$. Then for
every $Y$ separable subspace of $X$ containing $c$, there is a
linear and continuous projection $P_Y:Y\longrightarrow c$ with
$\Vert P\Vert\leq 8$. Indeed, let us consider the onto linear
isomorphism $T:c\longrightarrow c_0$ given by
$T(x)(1)=\frac{1}{2}\lim_n x(n)$ and
$T(x)(n)=\frac{1}{2}(x(n)-\lim_n x(n))$ for every $n>1$. Note that
$\Vert T\Vert =1$ and $\Vert T^{-1}\Vert=4$. Now, following
  \cite[Th. 5.14]{FHHMPZ}, we get the desired projection $P_Y$
  with $\Vert P_Y\Vert \leq 2\Vert T^{-1}\Vert =8$.

Let $\Upsilon$ be the family of subspaces $Y$ of $X$ containing
$c$ such that $c$ has finite codimension in $Y$. Consider the
filter basis $\Upsilon$ given by $\{Y\in \Upsilon: Y_0\subset
Y\}$, where $Y_0\in \Upsilon$ and call ${\mathcal U}$ the
ultrafilter containing the generated filter by the above filter
basis.

For every $Y\in \Upsilon$, we define a new norm in $X$ given by
$$\Vert x\Vert_Y:=\max\{\Vert P_Y(x)\Vert ,\Vert x-P_Y(x)\Vert\}.$$
Finally, we define the norm on $X$ given by $|\| x\| |
:=\lim_{\mathcal U}\Vert x\Vert_Y$. Observe that $\frac{1}{8}\Vert
x\Vert\leq |\| x\| |\leq 3\Vert x\Vert$ for every $x\in X$ and so
$|\|\cdot\| |$ is an equivalent norm in $X$ such that $|\| x\| |
=\Vert x\Vert_{\infty}$ for every $x\in c$, where
$\Vert\cdot\Vert_{\infty}$ is the sup norm in $c$. Hence $(X,|\|
\cdot\| |)$ contains an isometric copy of $c$.

Pick $x_0 \in B_{(X,|\| \cdot\| |)}$. In order to prove the
remaining statement let $\{e_n\}$ and $\{e_n^*\}$ the usual basis
of $c_0$ and the biorthogonal functionals sequence, respectively.

Chose $\lambda\in \real$ and $n\in \natu$. For every
$Y\in\Upsilon$ with $x_0\in Y$ we have that $$\Vert x_0+\lambda
e_n\Vert_Y =\max\{\Vert P_Y(x_0)+\lambda e_n\Vert ,\Vert
x_0-P_Y(x_0)\Vert\}=$$
$$\max\{\vert\lambda+e_n^*(P_Y(x_0))\vert ,\Vert
P_Y(x_0)-e_n^*(P_Y(x_0))e_n\Vert ,\Vert x_0-P_Y(x_0)\Vert\}.$$

Call  $\beta_n =\lim_{\mathcal U}\max\{ \Vert
P_Y(x_0)-e_n^*(P_Y(x_0))e_n\Vert ,\Vert x_0-P_Y(x_0)\Vert\}$
and\break $\alpha_n=\lim_{\mathcal U}e_n^*(P_Y(x_0))$. Then $\vert
\Vert x_0+\lambda e_n\vert \Vert=\max\{\vert \lambda
+\alpha_n\vert , \beta_n\}$. Note that $\vert \alpha_n\vert\leq 1$
and $\beta_n\leq 1$ since $|\| x_0\| |\leq 1$.

Doing $x_n:=x_0+(1-\alpha_n)e_n$ and $y_n:=x_0-(1+\alpha_n)e_n$
for every $n$, we get that $x_n$, $y_n\in B_{(X,|\| \cdot\| | )}$.
Finally, it is clear that $\{x_n\}$ and $\{y_n\}$ are weakly
convergent sequences to $x_0$ and $|\| x_n-y_n\| |=2$ for every
$n\in \natu$.\end{proof}

\begin{theorem}\label{c0} Let $X$ be a Banach space containing an isomorphic copy of
$c_0$. Then  there is an equivalent norm   in $X$ such
that:\begin{enumerate}\item[i)] Every slice of new unit ball of
$X$ has diameter $2$ for the new equivalent norm. \item[ii)] There
are nonempty relatively weakly open subsets of the new unit ball
in $X$ with diameter arbitrarily small for the new equivalent
norm.
\end{enumerate}\end{theorem}

\begin{proof}
From the above lemma, we can assume that $X$ contains an isometric
copy of $c$ and  for every $x\in B_{X}$ there are sequences
$\{x_n\}$, $\{y_n\}\in B_{X}$ weakly convergent to $x$ such that $
\Vert x_n -y_n\Vert =2$ for every $n\in \natu$.

Fix $0<\varepsilon<1$ and consider in $X$ the equivalent norm
$\Vert \cdot \Vert_{\varepsilon}$ whose unit ball is
$B_{\varepsilon}=\overline{co}(A\cup -A\cup [(1-\varepsilon)
B_X+\varepsilon B_{c_0}])$. Then we have $\Vert
x\Vert_{\varepsilon}\leq \frac{1}{1-\varepsilon}\Vert x\Vert$ for
every $x\in X$ and $\Vert x\Vert=\Vert x\Vert_{\infty}$ for every
$x\in c$.

In order to prove $ii)$, fix $\gamma >0$. Pick $n\in \natu$ with
$18<n (1-\varepsilon) \gamma$ and choose $\rho $ such that $0<\rho
<\frac{1}{n(32n-21)}$, $2\rho < \gamma$ and $2\rho <\varepsilon$.
Consider the relative weak open subset of $K$ given by
$$W_n=\{x\in K:e_{(0,i)}^*(x)>\frac{1}{n}-1-\rho,\ 1\leq i\leq n ,\
\lim_n x(n)<-1+\rho\}.$$ From Proposition \ref{Kw},
$W_n\neq\emptyset$ and $diam_{\Vert\cdot\Vert_{\infty}}(W_n)\leq
9/n$.

Now, we define $$W=\{x\in
B_{\varepsilon}:e_{(0,i)}^*(x)>\frac{1}{n}-1-\frac{1}{2}\rho ,\
1\leq i\leq n,\ \lim_n (x)<-1+\rho ^2 \},$$ where $e_n^*$ and
$\lim_n$ denote the Hanh-Banach extensions to $X$ of the
corresponding functionals on $c$. It is clear that $\Vert
e_{(0,i)}^*\Vert_{\varepsilon}=\Vert e_{(0,i)}^*\Vert =1$ for
every $i=1,\ldots , n$ and $\Vert \lim_n\Vert_{\varepsilon}=\Vert
\lim_n \Vert =1$.

We prove that $x_0=\sum_{i=1}^n \frac{x_{(0,i)}}{n}\in W$. For
this note that $\lim_n x_{(0,i)}(n)=-1$ and then $\lim_n
x_0(n)=-1<-1+\rho ^2$. Furthermore $x_0$ is a convex combination
of elements of $A$ and so $x_0\in B_\varepsilon$. Finally, for $
1\leq j\leq n$, one has that
$$e_{(0,j)}^*(x_0)=\sum_{i=1\ i\neq j}^n
\frac{1}{n}e_{(0,j)}^*(x_{(0,i)})+\frac{1}{n}
e_{(0,j)}^*(x_{(0,j)})=-\frac{n-1}{n}+\frac{1}{n}>\frac{1}{n}-1-\frac{1}{2}\rho.$$

Then $W$ is a nonempty relative weak open subset of
$B_{\varepsilon}$. In order to estimate the diameter of $W$, it is
enough compute the diameter of $W\cap co(A\cup -A\cup
[(1-\varepsilon) B_X+\varepsilon B_{c_0}])$. Furthermore,
$co(A\cup -A\cup [(1-\varepsilon) B_X+\varepsilon
B_{c_0}])=co(co(A)\cup co(-A)\cup [(1-\varepsilon) B_X+\varepsilon
B_{c_0}])$. So, given $x\in W$, we can assume that $x=\lambda
_1a+\lambda _2 (-b)+\lambda _3 [(1-\varepsilon )x_0+\varepsilon
y_0]$, where $\lambda _i\in [0,1]$ with $\sum_{i=1}^3\lambda_i=1$
and $a,b\in co(A)$, $x_0\in B_X$, and $y_0\in B_{c_o}$. Since
$x\in W$, we have that $\lim_n(x)<-1+\rho ^2$, and
hence,$$-\lambda _1+\lambda _2 +\lambda _3(1-\varepsilon
)\lim_n(x_0)= -\lambda _1+\lambda _2 +\lambda
_3\lim_n[(1-\varepsilon )x_0+\varepsilon y_0]<-1+\rho ^2.$$ Note
that $-1\leq \lim_n(x_0)$. This implies that
$$2\lambda _2+\lambda _3\varepsilon -1= -\lambda _1+\lambda _2
-\lambda _3(1-\varepsilon )<-1+\rho ^2.$$ Since $2\rho
<\varepsilon$, and so $\lambda _2+\lambda_3<\frac{1}{2}\rho$. As a
consequence we get that $\lambda_1>1-\frac{\rho}{2}$.

If $i\in\{1,\ldots ,n\}$ then
$$\frac{1}{n}-1-\frac{1}{2}\rho <e_{(0,i)}^*(x)=\lambda_1 e_{(0,i)}^*(a)+\lambda_2e_{(0,i)}^*(-b)+
\lambda _3e_{(0,i)}^*[(1-\varepsilon )x_0+\varepsilon y_0] .$$
Since $\Vert e_{(0,i)}^*\Vert _\varepsilon =1$ and
$-b,(1-\varepsilon )x_0+\varepsilon y_0\in B_\varepsilon$, we have
that $e_{(0,i)}^*(-b)\leq 1$ and $e_{(0,i)}^*[(1-\varepsilon
)x_0+\varepsilon y_0]\leq 1$. It follow that $$\lambda_1
e_{(0,i)}^*(a)+\lambda_2e_{(0,i)}^*(-b)+\lambda
_3e_{(0,i)}^*[(1-\varepsilon )x_0+\varepsilon y_0]\leq
$$$$\lambda_1 e_{(0,i)}^*(a)+\lambda_2+\lambda_3<\lambda_1
e_{(0,i)}^*(a)+\frac{1}{2}\rho .$$ We deduce that
$$e_{(0,i)}^*(\lambda_1 a)>\frac{1}{n}-1-\rho $$  for $1\leq i\leq n$. On the other hand, we have that
$$\lim_n(\lambda_1a)=-\lambda_1<-1+\frac{\rho}{2}<-1+\rho ,$$ and we
conclude that  $\lambda_1a\in W_n$.

Finally, given $x,x'\in W $, we can assume that
$$x=\lambda _1a+\lambda _2 (-b)+\lambda _3 [(1-\varepsilon )x_0+\varepsilon y_0], \ x'=\lambda _1'a'+\lambda _2' (-b')
+\lambda _3'[(1-\varepsilon )x_0'+\varepsilon y_0'],$$ where
$\lambda _i, \lambda _i'\in [0,1]$ with
$\sum_{i=1}^3\lambda_i=\sum_{i=1}^3\lambda_i'=1$, and
$a,b,a',b'\in co(A)$, $x_0,x_0'\in B_X$ and $y_0,y_0'\in B_{c_0}$.
We have that
$$\Vert x-x'\Vert_{\varepsilon}\leq \Vert
\lambda _1a-\lambda _1'a'\Vert_{\varepsilon}+
\lambda_2+\lambda_3+\lambda_2'+\lambda_3'<\Vert \lambda
_1a-\lambda _1'a'\Vert_{\varepsilon}+\rho .$$  We recall that,
$\Vert x\Vert_{\varepsilon}\leq \frac{1}{1-\varepsilon}\Vert
x\Vert$ for every $x\in X$ and $\Vert x\Vert=\Vert
x\Vert_{\infty}$ for every $x\in c$, and that
$\lambda_1a,\lambda_1'a'\in W_n$. Then
$$\Vert x-x'\Vert_{\varepsilon}\leq \frac{1}{1-\varepsilon}\Vert
\lambda_1a-\lambda_1'a'\Vert _\infty +\rho \leq
\frac{9}{n(1-\varepsilon )}+\rho\leq \gamma .$$  Hence
$diam_{\Vert\cdot\Vert_{\varepsilon}}(W)\leq \gamma.$

In order to prove i), note that $B_{\varepsilon}\subset B_X$ and
so $\Vert x\Vert_{\varepsilon} \geq \Vert x\Vert$ for every $x\in
X$.

Pick now $f\in X^*$, $\Vert f\Vert_{\varepsilon}^*=1$ and
$\beta>0$ and consider the slice $$S=\{x\in
B_{\varepsilon}:f(x)>1-\beta \}.$$

Hence there are $a\in A\cup -A$ or $(1-\varepsilon
)x_0+\varepsilon y_0\in (1-\varepsilon )B_X+\varepsilon B_{c_0}$
such that $a\in S$ or $(1-\varepsilon )x_0+\varepsilon y_0\in S$.

From the symmetry of $A\cup -A$, we can assume that $a\in A$, so
there is $\alpha\in \natu^{<\omega}$ such that $a=x_{\alpha}$. We
recall that  $x_{(\alpha,j)}(k)=1$ if $\phi^{-1}(k)\leq (\alpha,
j)$ and $x_{(\alpha,j)}(k)=-1$ in otherwise, then $\{x_{(\alpha
,j)}\}_j$ is a weakly convergent sequence to $x_{\alpha}$. Hence
we can choose $j$ so that $x_{(\alpha ,j)}\in S$. Note that
$x_{(\alpha ,j)}-x_{\alpha}= 2e_{(\alpha ,j)}$, then $2=\Vert
2e_{(\alpha ,j)}\Vert _\infty =\Vert x_{(\alpha
,j)}-x_{\alpha}\Vert \leq \Vert x_{(\alpha ,j)}-x_{\alpha}\Vert
_\varepsilon$. It follow that
$diam_{\Vert\cdot\Vert_{\varepsilon}}(S)= 2$.

In the case that there is $x_0\in B_X$  and $y_0\in B_{c_0}$ such
that $$(1-\varepsilon )x_0+\varepsilon y_0\in S,$$ as $S$ is a
norm open set, we can assume that $y_0$ has finite support. From
the above lemma, there is a scalars sequence $\{t_j\}$ with $\vert
t_j\vert \leq 1$ for every $j$ such that, putting
$x_j=x_0+(1-t_j)e_j$ and $y_j=x_0-(1+t_j)e_j$ for every $j$, we
have that $\{x_j\}$ and $\{y_j\}$ are weakly convergent sequences
in $B_X$ to $x_0$. We put $j_0$ such that $e_{j}^*(y_0)=0$ for
every $j\geq j_0$, then $y_0+e_j,y_0-e_j\in B_{c_0}$ for every
$j\geq j_0$.

So it follows that $\{ (1-\varepsilon )x_j+\varepsilon
(y_0+e_j)\}_{j\geq j_0}$ and $\{(1-\varepsilon )y_j+\varepsilon
(y_0-e_j)\}_{j\geq j_0}$ are sequences in $(1-\varepsilon
)B_{\varepsilon}+\varepsilon B_{c_0}\subset B_\varepsilon$ weakly
convergent to $(1-\varepsilon )x_0+\varepsilon y_0$. Hence we can
chose $j$ so that $(1-\varepsilon )x_j+\varepsilon (y_0+e_j)$, $
(1-\varepsilon )y_j+\varepsilon (y_0-e_j)\in S$. Then
$$\Vert [(1-\varepsilon )x_j+\varepsilon (y_0+e_j)]-[(1-\varepsilon )y_j+\varepsilon (y_0-e_j)]\Vert _\varepsilon
=$$ $$\Vert 2(1-\varepsilon)e_j+2\varepsilon e_j\Vert _\varepsilon
=\Vert 2e_j\Vert _\varepsilon \geq \Vert 2e_j\Vert =\Vert
2e_j\Vert _\infty =2,$$ and
$diam_{\Vert\cdot\Vert_{\varepsilon}}(S)=2$.
\end{proof}

As a consequence of the above result we have the following
corollary, which answers by the negative the problem about the
equivalence between SD2P and D2P.

\begin{corollary}\label{corc0} Every Banach space containing an isomorphic copy of $c_0$
can be equivalently renormed  satisfying SD2P and failing
D2P.\end{corollary}

Our final result establishes  a stability  property of Banach
spaces with SD2P and failing D2P by $\ell_2$-sums.

\begin{corollary}\label{corfin}  The $\ell_2$-sum of Banach spaces with SD2P and failing D2P also has SD2P and fails
D2P.
\end{corollary}

The proof of the above Corollary is an immediate consequence  of
the next lemma, which give us the stability of SD2P and small weak
open subsets for $\ell_2$-sums. In fact, this stability is also
true for $\ell_p$-sums, whenever $1\leq p<\infty$.

\begin{lemma}\label{lemfin} Let $\{X_n\}$ be a sequence of Banach spaces and
let $\{\varepsilon_n\}$ be a sequence of positive real numbers
with $\lim_n\varepsilon_n =0$. Assume that, for every $n\in
\natu$,  $X_n$  has SD2P and  $B_{X_n}$ contains a nonempty
relatively weakly open subset with diameter less than
$\varepsilon_n$. Then $\ell_2 -\bigoplus_n X_n$ has the SD2P and
the unit ball of $\ell_2 -\bigoplus_n X_n$ contains nonempty
relatively weakly open subsets with diameter arbitrarily
small.\end{lemma}

\begin{proof} For every $n\in \natu$ let $U_n$ be a nonempty
relatively weakly open subset of $B_{X_n}$ with diameter less than
$\varepsilon_n$. Then $\Vert x\Vert >1-\varepsilon_n$ for every
$x\in U_n$. Call $Z=\ell_2 -\bigoplus_n X_n$ and define, for $m\in
\natu$, $V_m=\{\{x_n\}\in B_Z: x_m\in U_m\}$. Now $V_m$ is a
relative weak open subset of $B_Z$. For fixed $m$, as $\Vert
x_m\Vert>1-\varepsilon_m$ we have that $\sum_{n=1\ n\neq
m}^{\infty}\Vert x_n\Vert^2\leq 1-(1-\varepsilon_m)^2$ for every
$\{x_n\}\in V_m$.

Pick $\{x_n\},\{y_n\}\in V_m$. So $\Vert x_m -y_m\Vert
<\varepsilon_m$, since $diam(U_m)<\varepsilon_m$ and then
$$\Vert\{x_n\}-\{y_n\}\Vert^2=\sum_{n=1}^{\infty}\Vert
x_n-y_n\Vert^2=$$
$$\sum_{n=1\ n\neq m}^{\infty}\Vert
x_n-y_n\Vert^2+\Vert x_m-y_m\Vert^2\leq$$
$$\sum_{n=1\ n\neq m}^{\infty}(\Vert
x_n\Vert+\Vert y_n\Vert)^2+\Vert x_m -y_m\Vert^2<$$
$$\sum_{n=1\ n\neq m}^{\infty}\Vert
x_n\Vert^2+\sum_{n=1 n\neq m}^{\infty}\Vert y_n\Vert^2+2\sum_{n=1
n\neq m}^{\infty}\Vert x_n\Vert \Vert y_n\Vert +\varepsilon_m^2
\leq$$
$$2(1-(1-\varepsilon_m)^2)+2(\sum_{n=1\ n\neq m}^{\infty}\Vert
x_n\Vert^2)^{1/2}(\sum_{n=1\ n\neq m}^{\infty}\Vert
y_n\Vert^2)^{1/2}+\varepsilon_m^2\leq$$
$$4(1-(1-\varepsilon_m)^2)+\varepsilon_m^2,$$
and $diam(V_m)\leq
(4(1-(1-\varepsilon_m)^2)+\varepsilon_m^2)^{1/2}$. As
$\lim_m\varepsilon_m=0$, we conclude that $B_Z$ has nonempty
relatively weakly open subsets with diameter arbitrarily small.

We pass now to prove that $Z$ has SD2P. Take $f\in S_Z$,
$0<\alpha<1$ and consider an arbitrary slice of $B_Z$ $$S=\{z\in
B_Z:f(z)>1-\alpha\}.$$ Pick $z_0\in S_Z\cap S$, then choose
$0<\varepsilon<\alpha$ so that $f(z_0)>1-\alpha+\varepsilon .$

We denotes by $P_n$ the projection of $Z$ onto $\ell_2
-\bigoplus_{i=1}^n X_i$, which is a norm one projection for every
$n\in \natu$. As $f(z_0)>1-\alpha+\varepsilon$, there is
$k\in\natu$ such that $P_k^*(f)(P_k(z_0))>1-\alpha+\varepsilon$,
where $P_k^*$ denotes the transposed  projection of $P_k.$

Consider the slice of the unit ball in $Y=\ell_2
-\bigoplus_{i=1}^k X_i$ given by $T=\{y\in
B_Y:P_k^*(f)(y)>1-\alpha+\varepsilon\}.$ In order to prove that
$diam(S)=2$, fix $\rho>0$ and take $y_1,y_2\in B_Y$ such that
$\Vert y_1 -y_2\Vert>2-\rho$. This is possible, because it is
known that the finite $\ell_2$-sum of Banach spaces with SD2P has
too SD2P \cite[Theorem 2.4]{AcBeLo}. Now we see $y_1, y_2$ as
elements in Z, via the natural isometric embedding of $Y$ into
$Z$, and we have that $y_1 ,y_2 \in S$ with $\Vert y_1 -y_2\Vert_Z
>2-\rho$, hence $diam(S)\geq 2-\rho$. As $\rho$ was arbitrary, we
conclude that $diam(S)=2$.\end{proof}

Now the proof of Corollary \ref{corfin} is complete.

It would be interesting to know if there is some Banach space with
PCP and SD2P.

On the other hand, there is a stronger property than D2P, the
strong diameter two property: a Banach space $X$ satisfies the
strong diameter two property (strong D2P) if every convex
combination of slices in the unit ball of $X$ has diameter 2. Its
clear that strong D2P implies D2P, and it is known that these two
properties are in fact different \cite{AcBeLo}. Indeed, in
\cite[Theorem 3.2]{AcBeLo}
  is proved   that $c_0\oplus_2 c_0$ is a Banach space with D2P and
failing strong D2P. The failure of strong D2P of $c_0\oplus_2 c_0$
is shown finding out an average of two slices in the unit ball
with diameter strictly less than 2, however the unit ball of
 $c_0\oplus_2 c_0$ has no convex combination of
slices with diameter arbitrarily small. In fact,  it is not
difficult to check that  the diameter of any convex combination of
slices in the unit ball of $c_0\oplus_2 c_0$ is at least $1$. So
it would be interesting to know if there is some Banach space with
D2P and so that its unit ball contains convex combinations of
slices with diameter arbitrarily small.

\end{document}